\documentclass[10pt]{article}

\usepackage{amsmath,amsfonts,latexsym, amsrefs,amssymb,amsthm}
\usepackage{dsfont}
\usepackage{stmaryrd}
\usepackage{wrapfig}
\usepackage{graphicx}
\usepackage{enumerate}
\usepackage{color}

\usepackage{mathrsfs}

\normalfont
\usepackage[T1]{fontenc}

\newtheorem{theorem}{Theorem}
\newtheorem*{theorem*}{Theorem}

\newtheorem{lemma}[theorem]{Lemma}
\newtheorem{proposition}[theorem]{Proposition}

\DeclareMathOperator{\tr}{Tr}
\DeclareMathOperator{\HH}{H}
\DeclareMathOperator{\RR}{\mathbb{R}}
\DeclareMathOperator{\NN}{\mathbb{N}}
\DeclareMathOperator{\ZZ}{\mathbb{Z}}

\DeclareMathOperator{\OO}{O}
\DeclareMathOperator{\oo}{o}
\DeclareMathOperator{\im}{Im}
\DeclareMathOperator{\re}{Re}

\newcommand{\E}{\mathbb{E}}

\newcommand{\1}{\mathds{1}}
\newcommand{\rd}{{\mathrm{d}}}            
\newcommand{\ii}{{\mathrm{i}}}

\newcommand{\e}{\varepsilon}
\renewcommand{\epsilon}{\varepsilon}

\title{Strong Szeg{\H o} asymptotics and zeros of the zeta function}

\date{}

\author{Paul Bourgade\quad
Jeffrey Kuan
 \\\\
Department of Mathematics, Harvard University\\
Cambridge MA 02138, USA \\   bourgade@math.harvard.edu \quad
 jkuan@math.harvard.edu
}

\begin{document}

\maketitle


\begin{abstract}
Assuming the Riemann hypothesis, we prove the weak convergence of linear statistics of the zeros of L-functions to a Gaussian field,
with covariance structure corresponding to the $\HH^{1/2}$-norm of the test functions.
For this purpose, we obtain an approximate form of the explicit formula, relying on
Selberg's smoothed expression for $\zeta'/\zeta$ and the Helffer-Sj\"ostrand functional calculus.
Our main result is  an analogue of the strong Szeg{\H o} theorem, known for Toeplitz operators and random matrix theory.
\end{abstract}


\medskip

\medskip


\section{Introduction}

A connection between $\zeta$ zeros and random matrix theory was discovered
by Montgomery \cite{Mon1973}, who  examined
 the pair correlation of the zeta zeros. Dyson was the first to notice that this pair correlation agrees with the pair correlation of the eigenvalues of stochastic Hermitian matrices
with properly distributed Gaussian entries.
Assuming the Riemann hypothesis  (we denote by $1/2\pm\ii\gamma_j$, $\gamma_j\in\RR$, $0\leq \gamma_1\leq \gamma_2\leq \ldots$,
the set of non-trivial zeros), Montgomery proved that
$$
\frac{1}{x}\sum_{1\leq j,k\leq x, j\neq k}f(\tilde\gamma_j-\tilde\gamma_k)\underset{x\to
\infty}{\longrightarrow}\int_{-\infty}^{\infty}f(y) \left(1-\left(\frac{\sin{\pi y}}{\pi y}\right)^2\right)
\rd y ,
$$
where the $\tilde\gamma$'s are the rescaled zeta zeros ($\tilde\gamma=\frac{\gamma}{2\pi}\log\gamma$: at height $t$,
the average gap between zeros is $2\pi/\log t$), and the  test function $f$ has a smooth
Fourier transform supported in $(-1,1)$. A fundamental conjecture in analytic number theory concerns
removing this support condition.  This would imply, for example, new estimates on large gaps between primes,
but it seems out of reach with available techniques. In particular, this requires a better understanding of some asymptotic correlations between primes,  such as the Hardy-Littlewood conjectures, as shown in \cite{BogKea1996}.
 Further examples of this connection
appear in \cite{RudSar1996} for the correlation functions of order greater than 2, in
\cite{KatSar1999} for the function-field L-functions,
and in \cite{KeaSna2000I} for the conjectured asymptotics of the moments of $\zeta$ along the critical axis. 

 By looking at linear statistics, Hughes and Rudnick \cite{HugRud2002} demonstrate another way  
to exhibit the repulsion between the $\zeta$ zeros at the microscopic scale.
They showed that if the function $f$ has a smooth Fourier transform supported on $(-2/m,2/m)$, then the first $m$ moments
of the linear statistics (here and in the following $\omega$ is a uniform random variable on $(1,2)$)
\begin{equation}\label{eqn:HR}
\sum_{\gamma} f(\tilde\gamma-\omega t)
\end{equation} 
converge\footnote{By analogy with what is known in random matrix theory \cite{HugRud2003}, the higher moments supposedly do not converge towards those of this Gaussian random variable (see also \cite{HugMil2007} for a similar rigorous fact about non-Gaussianness in the context of low--lying zeros of L--functions).} to those of a Gaussian random variable as $t\to\infty$.
We propose to look at linear statistics at a larger (mesoscopic) scale.

Contrary to the Dyson-Montgomery analogy, observed at the microscopic level of nearest zeros spacings,
the mesoscopic regime involves a larger window and  yields
Gaussian fluctuations.
Indeed, Selberg proved, unconditionally, 
the following central limit theorem \cite{Sel1944,Sel1946,Sel1992}:
if $\omega$ is uniform on (1,2), as $t\to\infty$,
$$
\frac{\log \zeta\left(\frac{1}{2}+ \ii \omega t\right)}{\sqrt{\log\log t}}\to \mathcal{N}_1+ \ii\mathcal{N}_2,
$$
where $\mathcal{N}_1$ and $\mathcal{N}_2$ are independent standard normal random variables.
This is related to the fluctuations between the number of zeros with imaginary part in $[0,t]$ and their expected number.
The very small normalization in this convergence in law indicates the repulsion of the zeros. 
This central limit theorem was extended by Fujii to the fluctuations when counting zeros in smaller (but still mesoscopic) intervals
\cite{Fuj1999}.
Central limit
theorems concerning counting the number of eigenvalues of random matrices appeared originally  in \cite{CosLeb1995} for Gaussian ensembles and \cite{KeaSna2000I, Wie2002} for the circular unitary ensemble.

In this paper, we extend  (conditioned on the Riemann hypothesis)
these results on Gaussian fluctuations of zeros of L-functions to smoother statistics than
indicator functions of intervals. This includes an analogue of the strong Szeg{\H o} theorem,
seen originally as the second-order asymptotics of Toeplitz determinants as the dimension increases.
It is also related, by Heine's formula, to linear statistics of eigenangles of 
Haar-distributed unitary matrices, $S_n(f)=\sum_1^n f(\theta_n)$. Indeed,  for $f$ with mean $0$ on $[0,2\pi]$ satisfying $f(0)=f(2\pi)$, and
$\lambda\in\mathbb{R}$, then the strong Szeg{\H o} theorem states that
$$
\mathbb{E}\left(e^{\lambda S_n(f)}\right)
\underset{n\to\infty}{\longrightarrow}
\exp\left(\frac{1}{2}\lambda^2\sum_{k=-\infty}^\infty \vert k \vert \cdot |\hat{f}_k|^2\right),
$$
where the $\hat{f}_k$'s are the Fourier coefficients of $f$ ($\hat f_j=\frac{1}{2\pi}\int f(\theta)e^{-\ii j\theta}\rd\theta$).
In probabilistic terms, the convergence of the above Laplace transform means that the
linear statistics of the eigenvalues converge with no normalization to a normal random variable with
variance $\sum_{\ZZ} |k| |\hat{f}_k|^2$; the only restriction is that this limiting variance is finite. This was extended by Johansson in the context of
Coulomb gases on the unit circle \cite{Joh1988} and on the real line \cite{Joh1998}. Other proofs of the strong Szeg{\H o} theorem were given,
relying for example on combinatorics \cite{Kac1954},
on representation theory \cite{DiaSha1994,DiaEva2001},
on the steepest descent method for Riemann-Hilbert problems \cite{Dei1999},
on the Borodin-Okounkov formula \cite{BorOko2000} (see \cite{Sim2005} for many on these distinct proofs).

These linear statistics asymptotics were extended by Diaconis and Evans \cite{DiaEva2001},
to the general setting of more irregular test functions.
Under the hypothesis $f\in L^2(\mathbb{T})$, denoting $\sigma_n^2=
\sum_{j=-n}^n|j||\hat f_j|^2$, they proved that if $(\sigma_n)_{n\geq 1}$ is slowly varying then,
as $n\to\infty$,
$\frac{S_n-\E(S_n)}{\sigma_n}$ converges in distribution to a standard normal random variable.
This wide class of possible test functions  includes the smooth and indicator cases.
For many determinantal point processes,
a similar central limit theorem was obtained by Soshnikov under weak assumptions on the regularity of $f$ \cite{Sos2002}.
Moreover, for smoother test functions $f$,  he proved a local version of the strong Szeg{\H o} theorem \cite{Sos2000}: 
the linear statistics are of type (e.g. for the unitary group) $\sum_{k=1}^n f(\lambda_n \theta_k)$,
for a parameter $\lambda_n$ satisfying\footnote{In the following $a\ll b$ means $a=\oo(b)$.} $1\leq \lambda_n\ll n$. The last inequality means that we keep in the mesoscopic regime.

Our purpose consists of an analogue of the above results for linear functionals of zeros of the zeta function. This concerns linear statistics of type
\begin{equation}\label{Sz}
\sum_\gamma f (\lambda_t(\gamma-\omega t)),
\end{equation}
where
$\omega$ is uniform on $(1,2)$, as in (\ref{eqn:HR}), but now the condition $\lambda_t\ll \log t$ gives the mesoscopic regime: the number of
zeros visited by $f$ goes to infinity.

In the following statements,  $\omega$ is uniform on $(1,2)$,  we denote by $\{1/2+\ii\gamma\}$ ($\gamma\in\RR$,
we assume the Riemann hypothesis)
the multiset of non-trivial zeros
of $\zeta$, counted with repetition. We define $\gamma_t=\lambda_t(\gamma-\omega t)$ 
and $\sigma_t(f)^2=\int_{-(\log t)/\lambda_t}^{(\log t)/\lambda_t}|u| |\hat f(u)|^2\rd u$, where $\hat f(u)=\frac{1}{2\pi}\int
f(x)e^{-\ii u x}\rd x$.
Moreover, the centered, normalized linear statistics are  denoted
$$S_t(f)=
\sum_\gamma f(\gamma_t)-\frac{\log t}{2\pi\lambda_t}\int f(u)\rd u.$$
Our first result states that, for functions with sufficient regularity, the linear statistics converge to a Gaussian field with covariance function given by ($f,g$ are real functions, for notational simplicity)
$$
\langle f,g\rangle_{\HH^{1/2}}=\Re\int_{\RR}|u| \hat f(u)\overline{\hat g(u)}\rd u=
-\frac{2}{\pi^2}\int f'(x)g'(y)\log|x-y|\rd x\rd y,
$$
where we refer to \cite{DysMeh63} equation (18) for the last equality. Our technical assumptions on $f$ are the  following: sufficient decay of $f$ at $\pm\infty$, bounded variation of $f$, and sufficient decay of $\hat{f}$ at $\pm\infty$. More specifically,
\begin{align}
\label{eqn:decrease}
&\text{for some}\ \delta>0\ \text{and}\ \vert x\vert\ \text{large enough}, \ f(x), f'(x), f''(x)\ \text{exist and are}\ \OO\left(x^{-2-\delta}\right),\\
&g(x):=f'(x),\ \int(1+|u\log u|)|\rd g(u)|<\infty,\label{eqn:cond1}\\
&\xi|\hat f(\xi)|^2,\,(\xi|\hat f(\xi)|^2)'=\OO(\xi^{-1}).\label{eqn:Fourier}
\end{align}
Our assumptions on $f$  easily include the cases of compactly supported $\mathscr{C}^2$ functions, for example.
We will also assume that $\|f\|_{\HH^{1/2}}<\infty$, and
note that, as discussed in \cite{DiaEva2001}, there is no good characterization of the space $\mathcal{H}^{1/2}
=\{\|f\|_{\HH^{1/2}}<\infty\}$ in terms of the local regularity of $f$. In particular, it is likely that our
assumption (\ref{eqn:cond1}) may be slightly relaxed. 

The assumption (\ref{eqn:Fourier}) appears necessary in 
a second moment calculation (Lemma
\ref{Second Moment}), and there is no analogous restriction in the case of random matrices \cite{DiaEva2001}; 
it is certainly possible to slightly weaken it but we do not pursue this goal here, as (\ref{eqn:Fourier}) obviously already allows smooth functions but also indicators.

\begin{theorem}\label{thm:1} 
Let $f_1,\dots,f_k$ be functions in $\mathcal{H}^{1/2}$ satisfying properties (\ref{eqn:decrease}), (\ref{eqn:cond1}), (\ref{eqn:Fourier}). Assume the Riemann hypothesis and that $1\ll\lambda_t\ll \log t$.
Then the random vector $(S_t(f_1),\dots,S_t(f_k))$
converges in distribution to a centered Gaussian vector $(S(f_1),\dots,S(f_k))$ with correlation structure.
$$
\E(S(f_h)S(f_\ell))=\langle f_h,f_\ell\rangle_{\HH^{1/2}}.
$$
\end{theorem}

The absence of normalization for the above convergence in law is a tangible sign of 
repulsion between the $\zeta$ zeros.
 However, there are differences between our result and the strong Szeg{\H o}  theorem: in particular, the rate of convergence to the limiting Gaussian is expected to be slow in our situation, while it is extremely fast in the case of random unitary  matrices \cite{Joh1997}.

In the following theorem, for diverging variance of linear statistics, the bounded variation assumption is weaker:
\begin{equation}
\label{eqn:cond2}
\int(1+|u\log u|)|\rd f(u)|<\infty.
\end{equation}

\begin{theorem}\label{thm:2}
Suppose $f$ satisfies (\ref{eqn:decrease}), (\ref{eqn:cond2}), (\ref{eqn:Fourier}), and that
$\sigma_t(f)$ diverges. Assume the Riemann hypothesis and that $1\ll\lambda_t\ll \log t$. Then, as $t\to\infty$, $S_t(f)/\sigma_t(f)$
converges in distribution to a standard Gaussian random variable.
\end{theorem}

Although there is a normalization in the above theorem, it is typically  very small. For example, in the
allowed case when $f$ is an indicator function and $\lambda_t$  grows very slowly, $\sigma_t^2$ will be of order $\log\log t$, agreeing with\ the central limit theorem
proved (unconditionally) by Selberg.

As we already noted, the condition $\lambda_t\ll \log t$ implies the mesoscopic scale. The condition $1\ll \lambda_t$ is less natural.
Supposedly, asymptotic normality does not hold if $\lambda_t=\OO(1)$, for test functions in $\HH^{1/2}$. This is related to the phenomenon of variance saturation explained by Berry \cite{Ber1988}, which happens at the same transition of the parameter  $\lambda_t$. Motivated by this result, Johansson exhibited determinantal point processes satisfying the same phenomenon \cite{Joh2004}. 
Note that such a transition, where limiting normality fails, also appears for sums of random exponentials \cite{BenBogMol2005},
in particular for the Random Energy Model. The ultrametric structure for this model also appears for the counting measure of the $\zeta$ zeros \cite{Bou2010}. Interesting conjectures relating long-range dependence particle systems and 
extreme values of L-functions were developed in \cite{FyoHiaKea2012}.

The technique employed in the proof of both theorems involves an approximate version of the Weil explicit formula relating 
the zeros and primes  (Section 2). This uses the Helffer-Sj\"ostrand functional calculus, which enables us to consider non-analytic test functions, and Selberg's seminal formula for $\zeta'/\zeta$.

Finally, we want to mention that while finishing this manuscript we discovered,
in the draft  \cite{Rod2012}, 
a preliminary proof of Theorem \ref{thm:2} 
which seems not to require Selberg's formula (\ref{eqn:Selberg}).

\section{Approximate explicit formula}

In this section, we consider a function $f:\RR\to\RR$ of class $\mathscr{C}^2$, satisfying
(\ref{eqn:decrease})
as $x\to\pm \infty$. We aim at proving the following approximate version of the Weil explicit
formula, relying on
Selberg's smoothed expression for $\zeta'/\zeta$ and the Helffer-Sj\"ostrand functional calculus.
Remember that $\omega$ is a uniform random variable on $(1,2)$, $\gamma_t=\lambda_t(\gamma-\omega t)$,
and we use Selberg's smoothed  von Mangoldt function,
$$
\Lambda_u(n)=\left\{
\begin{array}{ll}
\Lambda(n)&{\ {\rm if}\ }1\leq n\leq u,\\
\Lambda(n)\frac{\log(u^2/n)}{\log n}&{\ {\rm if}\ }u\leq n\leq u^2,\\
0&{\ {\rm otherwise}\ }.
\end{array}
\right.
$$

\begin{proposition}\label{prop:appExpl}
Assume the Riemann hypothesis. For any $f\in \mathscr{C}^2$ satisfying the initial assumptions, and any $u=t^\alpha$, $\alpha>0$ fixed,
$$
\sum_\gamma f(\gamma_t)
-\frac{\log t}{2\pi\lambda_t}\int f
=\frac{1}{\lambda_t}\sum_{n\geq 1}\frac{\Lambda_u(n)}{\sqrt{n}}\left(
\hat f\left(\frac{\log n}{\lambda_t}\right)n^{\ii \omega t}
+\hat f\left(-\frac{\log n}{\lambda_t}\right)n^{-\ii \omega t}
\right)+E(\omega,t),
$$
where the chosen Fourier normalization is $\hat f(\xi)=\frac{1}{\pi}\int f(x)e^{-\ii\xi x}\rd x$ and the error term $E(\omega,t)$ is of type
\begin{equation}\label{eqn:error}
X(\omega,t)\frac{\lambda_t}{\log t}\OO\left(\|f\|_{1}+\|f'\|_{1}+\|f''\|_{1}+\frac{1}{t\log t}\|x\log x f\|_1+
\frac{1}{t\log t}\|x\log x f'\|_1
+
\frac{1}{t\log t}\|x\log x f''\|_1\right),
\end{equation}
where $\E(|X(\omega,t)|)$ is uniformly bounded and does not depend on $f$. 
\end{proposition}

It is clear that if  $f$ is a fixed compactly supported $\mathscr{C}^2$ function, the error term converges in probability to $0$ as $t\to\infty$. However, in our application of Proposition \ref{prop:appExpl}, $f$ can depend on $t$.

Moreover, we state this approximate version of the explicit formula in a probabilistic setting for convenience,
as this is what is needed in the proof of Theorems \ref{thm:1} and \ref{thm:2}. One could also state a deterministic
version, for functions with compact support along the critical axis.

\begin{proof}
All of the integrals in $\rd x\rd y$ in this paper are on the domain $\mathcal{D}:=\{x\in\RR,y>0\}$. The following formula, from the Helffer-Sj\"ostrand functional calculus, will be useful for us: for any $q\in\RR$,
\begin{equation}\label{eqn:HS}
f(q)=\Re\left(\frac{1}{\pi}\iint_{\mathcal{D}}\frac{\ii y f''(x)\chi(y)+\ii (f(x)+\ii y f'(x))\chi'(y)}{q-(x+\ii y)}\rd x\rd y\right),
\end{equation}
where $\chi$ is a smooth cutoff function equal to $1$ on $[0,1/2]$, $0$ on $[1,\infty)$.
This is one of many possible formulas aiming originally at evaluating $\tr f(H)$
from resolvent estimates of $H$, for general self-adjoint operators $H$ and test function $f$  (see e.g.\cite{Dav1995}). We follow this idea here, the {\it resolvent estimate} being Selberg's expression for $\zeta'/\zeta$.

Let $\gamma_t=\lambda_t(\gamma-\omega t)$, and $N(t)$ be the number of $\gamma$'s in $[0,t]$ (counted with multiplicity). It is well-known (see e.g. \cite{Tit1986}) that, as $t\to\infty$,
\begin{equation}\label{eqn:Nt}
N(t)=\frac{t}{2\pi}\log t-\frac{1+\log(2\pi)}{2\pi} t+\OO(\log t).
\end{equation}
From (\ref{eqn:HS}), taking real parts, we obtain (here $z=x+\ii y$)
\begin{align}
\sum_{|\gamma|<M} f(\gamma_t)=&
-\label{term1}
\frac{1}{\pi}\iint_{\mathcal{D}} y f''(x)\chi(y)\sum_{|\gamma|<M}\im\left(
\frac{1}{\gamma_t-z}\right)\rd x\rd y\\
&-\label{term2}
\frac{1}{\pi}\iint_{\mathcal{D}} f(x)\chi'(y)\sum_{|\gamma|<M}\im\left(
\frac{1}{\gamma_t-z}\right)\rd x\rd y\\
&-
\frac{1}{\pi}\iint_{\mathcal{D}} y f'(x)\chi'(y)\sum_{|\gamma|<M}\re\left(
\frac{1}{\gamma_t-z}-\frac{1}{\lambda_t}\frac{\gamma}{\gamma^2+\frac{1}{4}}\right)\rd x\rd y\label{term3}
\end{align}
(it will soon be clear why we add the $\gamma/(\gamma^2+1/4)$ term, which makes no contribution in the integral).
We now prove that  by dominated convergence, the above three terms converge as $M\to\infty$.
First, note that $y\mapsto y\im((\gamma-(x+\ii y))^{-1})$ is increasing, so using (\ref{eqn:decrease}),
$$
(\ref{term1})
\leq
\int |f''(x)|\sum_{\gamma}\frac{1}{1+(\gamma_t-x)^2}\rd x
\leq
\int \sum_{\gamma}\frac{1}{1+x^2}\frac{1}{1+(\gamma_t-x)^2}\rd x
\leq
\sum_{\gamma}\frac{1}{1+\gamma_t^2}<\infty,
$$
where we used
\begin{equation}\label{eqn:convolution}
\int\frac{1}{1+(a-x)^2}\frac{1}{1+x^2}\rd x\leq\frac{1}{1+a^2},
\end{equation}
where all the above (and following) inequalities are up to universal constants. Moreover $\chi'$ is supported on [1/2,1],
still using (\ref{eqn:decrease}) and (\ref{eqn:convolution}) it is immediate that (\ref{term2}) converges as well.
Finally, concerning (\ref{term3}),  grouping for example $\gamma$ with $-\gamma$ in the sum,
we can bound it by
\begin{multline*}
\iint_{\mathcal{D}} y \frac{1}{1+x^2}|\chi'(y)|\sum_{0\leq \gamma\leq M}\left|\re\left(\frac{1}{\gamma_t-z}+
\frac{1}{(-\gamma)_t-z}\right)
\right|\rd x\rd y\\
\leq
\int \frac{1}{1+x^2}\sum_{0\leq \gamma\leq M}\left|\re\left(\frac{1}{\gamma_t-(x+\ii)}+
\frac{1}{(-\gamma)_t-(x+\ii)}\right)
\right|\rd x
\end{multline*}
and it is an integration exercise to prove that the contribution of each $\gamma$  in this integral
is $\OO(\gamma^{-3/2})$ for example.  Note that
from (\ref{eqn:decrease}) and (\ref{eqn:Nt}), $\sum_\gamma f(\gamma_t)$ is absolutely summable for each fixed $t$.
We therefore proved that
\begin{align}
\sum_{\gamma} f(\gamma_t)=&
-\label{tterm1}
\frac{1}{\pi}\iint_{\mathcal{D}} y f''(x)\chi(y)\sum_{\gamma}\im\left(
\frac{1}{\gamma_t-z}\right)\rd x\rd y\\
&-\label{tterm2}
\frac{1}{\pi}\iint_{\mathcal{D}} f(x)\chi'(y)\sum_{\gamma}\im\left(
\frac{1}{\gamma_t-z}\right)\rd x\rd y\\
&-
\frac{1}{\pi}\iint_{\mathcal{D}} y f'(x)\chi'(y)\sum_{\gamma}\re\left(
\frac{1}{\gamma_t-z}-\frac{1}{\lambda_t}\frac{\gamma}{\gamma^2+\frac{1}{4}}\right)\rd x\rd y\label{tterm3},
\end{align}
where all sums are absolutely convergent.
Now, the above sums can be written in terms of $\zeta'/\zeta$: it is known from
Hadamard's factorization formula that,  denoting $\rho$'s for the non-trivial $\zeta$-zeros, for any $s\not\in\{\rho\}$, we have
(see e.g. p 398 in \cite{MonVau2007})
$$
\frac{\zeta'}{\zeta}(s)=-\frac{1}{s-1}+\sum_\rho\left(\frac{1}{s-\rho}+\frac{1}{\rho}\right)-\frac{1}{2}\log\im(s)
+\OO(1).
$$
By taking real and imaginary parts, and identifying $s=1/2+\frac{y}{\lambda_t}+\ii \left(\omega t+\frac{x}{\lambda_t}\right)$, we get 
\begin{align*}
\sum_{\gamma}\im\left(
\frac{1}{\gamma_t-(x+\ii y)}\right)&=
\frac{1}{\lambda_t}\re\frac{\zeta'}{\zeta}\left(\frac{1}{2}+\frac{y}{\lambda_t}+\ii \left(\omega t+\frac{x}{\lambda_t}\right)\right)
+\frac{1}{2}\frac{\log t}{\lambda_t}
+\frac{1}{\lambda_t}\OO\left(\left|\log\left(\omega +\frac{x}{t\lambda_t}\right)\right|\right),
\\
\sum_{\gamma}\re\left(
\frac{1}{\gamma_t-(x+\ii y)}-\frac{1}{\lambda_t}\frac{\gamma}{\gamma^2+\frac{1}{4}}\right)
&=
\frac{1}{\lambda_t}\im\frac{\zeta'}{\zeta}\left(\frac{1}{2}+\frac{y}{\lambda_t}+\ii \left(\omega t+\frac{x}{\lambda_t}\right)\right)
+\OO\left(\frac{1}{\lambda_t}\right).
\end{align*}

Still relying on (\ref{eqn:HS}), using the fact that
$
\lim_{M\to\infty}\im\int_{-M}^M
\frac{\rd u}{u-{x+\ii y}}=\int\frac{\rd v}{x^2+1}=\pi,\
\lim_{M\to\infty}\re\int_{-M}^M
\frac{\rd u}{u-{x+\ii y}}=0,
$
we have
$$
\frac{\log t}{2\pi\lambda_t}\int f(u)\rd u=
-\frac{1}{\pi}\int y f''(x)\chi(y)\frac{1}{2}\frac{\log t}{\lambda_t}\rd x\rd y\\
-
\frac{1}{\pi}\int f(x)\chi'(y)\frac{1}{2}\frac{\log t}{\lambda_t}\rd x\rd y,
$$
so we obtained,
\begin{align}
\sum_{\gamma} f(\gamma_t)-\frac{\log t}{2\pi\lambda_t}\int f=&\notag
-\frac{1}{\pi\lambda_t}\iint_{\mathcal{D}} y f''(x)\chi(y)
\re\frac{\zeta'}{\zeta}\left(\frac{1}{2}+\frac{y}{\lambda_t}+\ii \left(\omega t+\frac{x}{\lambda_t}\right)\right)
\rd x\rd y\\
&\notag
-\frac{1}{\pi\lambda_t}\iint_{\mathcal{D}} f(x)\chi'(y)
\re\frac{\zeta'}{\zeta}\left(\frac{1}{2}+\frac{y}{\lambda_t}+\ii \left(\omega t+\frac{x}{\lambda_t}\right)\right)
\rd x\rd y\\
&
-\frac{1}{\pi\lambda_t}\iint_{\mathcal{D}} y f'(x)\chi'(y)
\im\frac{\zeta'}{\zeta}\left(\frac{1}{2}+\frac{y}{\lambda_t}+\ii \left(\omega t+\frac{x}{\lambda_t}\right)\right)
\rd x\rd y\label{eqn:fluct1}
+\OO\left(\frac{1}{\lambda_t}\right),
\end{align}
where the above $\OO(\lambda_t^{-1})$ is understood in the sense that its $L^1$ norm is bounded by $\lambda_t^{-1}$.
We now substitute $\frac{\zeta'}{\zeta}$, in the above expression, with its smooth approximation by Selberg: for any $u>0$
and $s\not\in\{\rho,1, -2\NN\}$,
\begin{equation}\label{eqn:Selberg}
\frac{\zeta'}{\zeta}(s)=A_u(s)+B_u(s)+C_u(s)+D_u(s)
\end{equation}
where
\begin{align*}
A_u(s)&=-\sum_{n\leq u^2}\frac{\Lambda_u(n)}{n^s},\\
B_u(s)&=
\frac{1}{\log u}\sum_{\rho}\frac{u^{\rho-s}-u^{2(\rho-s)}}{(\rho-s)^2},\\
C_u(s)&=
\frac{1}{\log u}\sum_{n\geq 1}\frac{u^{-2n-s}-u^{-2(2n+s)}}{(2n+s)^2},\\
D_u(s)&=
\frac{1}{\log u}\frac{u^{2(1-s)}-u^{1-s}}{(1-s)^2}.
\end{align*}
First, it is elementary that the contribution from the terms $D_u$ and $C_u$ in (\ref{eqn:fluct1}) is negligible.
For $D_u$, we bound by
$
\frac{1}{\log u}\int \frac{u}{1+(\omega t)^2}\frac{\rd x}{1+x^2}\leq
\frac{1}{\log u} \frac{u}{1+t^2}\leq \frac{c}{t},
$
under the constraint $1\ll u\leq t$ (by the end we will choose $u=t^{1/2}$).
 The term involving $C_u$ is also $\OO(t^{-1})$ easily.

The main errors involve $B_u$.
First, as $\chi'$ is supported on $(1/2,1)$, we have
\begin{multline}
\frac{1}{\lambda_t}\iint_{\mathcal{D}}(|f(x)\chi'(y)|+|y f'(x)\chi'(y)|)\left|B_u\left(\frac{1}{2}+\frac{y}{\lambda_t}+\ii \left(\omega t+\frac{x}{\lambda_t}\right)\right) \right|\rd x\rd y\\
\leq
\frac{1}{\lambda_t\log u}e^{-\frac{\log u}{2\lambda_t}}\int
(|f(x)|+|f'(x)|)\sum_{|\gamma|<4t+\frac{4|x|}{\lambda_t}}\E
\frac{1}{(1/\lambda_t)^2+(\omega t-\gamma+x/\lambda_t)^2}
\rd x\\
+
\frac{1}{\lambda_t\log u} e^{-\frac{\log u}{2\lambda_t}}\int
 (|f(x)|+|f'(x)|)\sum_{|\gamma|>4t+\frac{4|x|}{\lambda_t}}\E
\frac{1}{(1/\lambda_t)^2+(\omega t-\gamma+x/\lambda_t)^2}
\rd x
\label{eqn:farFromAxis}
\end{multline}
Using (\ref{eqn:Nt}), the first sum is at most
$$
\frac{|\{|\gamma|<4t+\frac{4|x|}{\lambda_t}|\}|}{t}
\int \frac{ \rd v}{(1/\lambda_t)^2+v^2}
\leq
(\log t+\frac{\vert x\vert\log \vert x\vert }{t})\lambda_t,
$$
and the second at most $\sum 1/\gamma^2<\infty$, so this error is of type (\ref{eqn:error}),
for  $u=t^\alpha$.
Finally, the error from $B_u$ in the expectation of the term (\ref{tterm1}),  which is closer to the critical axis, is bounded by
\begin{align*}
&\frac{1}{\lambda_t\log u}\iint_{\mathcal{D}}
y e^{-\frac{y}{\lambda_t}\log u}|f''(x)|\sum_{|\gamma|<4t+\frac{4|x|}{\lambda_t}}\E
\frac{1}{(y/\lambda_t)^2+(\omega t-\gamma+x/\lambda_t)^2}
\rd x\rd y\\
+&
\frac{1}{\lambda_t\log u}\iint_{\mathcal{D}}
y e^{-\frac{y}{\lambda_t}\log u} |f''(x)|\sum_{|\gamma|>4t+\frac{4|x|}{\lambda_t}}\E
\frac{1}{(y/\lambda_t)^2+(\omega t-\gamma+x/\lambda_t)^2}
\rd x\rd y.
\end{align*}
The first sum is at most
$$
\frac{|\{|\gamma|<2t+\frac{2|x|}{\lambda_t}\}|}{t}
\int \frac{1}{(y/\lambda_t)^2+u^2} \rd u
\leq
(\log t+\frac{\vert x\vert\log \vert x\vert}{t})\frac{\lambda_t}{y},
$$
and the second at most $\sum 1/\gamma^2<\infty$, so all together this error term is of type (\ref{eqn:error}).

Finally, the $A_u(s)$ term can be simplified observing, by successive integrations by parts
\footnote{
In detail,
\begin{align*}
& \frac{1}{\pi}\int f''(x)e^{-\ii\delta x}\rd x\int y\chi(y)e^{-\delta y}\rd y = \frac{1}{\pi}\int \ii\delta y\chi(y)e^{-\delta y}\rd y\int f'(x)e^{-\ii\delta x}\rd x\\
=& \frac{1}{\pi}\int (\ii\chi(y)+\ii y\chi'(y))e^{-\delta y}\rd y\int f'(x)e^{-\ii\delta x}\rd x.
\end{align*}
And notice that the $\ii y\chi'(y)$ term cancels, and the other term equals
$$
\frac{1}{\pi}\int \ii\chi(y)e^{-\delta y}\rd y\int f'(x)e^{-\ii\delta x}\rd x = -\frac{1}{\pi}\int \delta\chi(y)e^{-\delta y}\rd y\int f(x)e^{-\ii\delta x}\rd x.
$$
and a final integration by parts gives (\ref{eqn:integration by parts}).
}, that for any $\delta>0$ we have
\begin{equation}\label{eqn:integration by parts}
\frac{1}{\pi}\int (y f''(x)\chi(y)+ (f(x)-\ii y f'(x))\chi'(y))e^{-\ii\delta x}e^{-\delta y}\rd x\rd y
=
-\frac{1}{\pi}\int f(x)e^{-\ii\delta x}\rd x.
\end{equation}
This completes the proof of Proposition \ref{prop:appExpl}.
\end{proof}

\section{Strong Szeg{\H o} theorem}

We first prove that, in Proposition \ref{prop:appExpl},
the terms $n$ of type $p^k$, for $k\geq 2$, make no contribution.

\begin{lemma} \label{lem:largepowers}For $u=t^\alpha$, $\alpha\leq 1$, and a family of functions $(f_t)$ uniformly bounded in $L^1$,
the random variable
$$
\frac{1}{\lambda_t}\sum_{n=p^k,p\in\mathcal{P},k\geq 2}\frac{\Lambda_u(n)}{\sqrt{n}}\hat f_t\left(\frac{\log n}{\lambda_t}\right)n^{\ii \omega t}
$$
converges to 0 in $L^2$.
\end{lemma}

\begin{proof}
For the terms corresponding to $k\geq 3$, this is obvious by absolute  summability.
For $k=2$, we
 can use the Montgomery-Vaughan inequality \cite{MonVau1974}:
for any complex numbers $a_r$ and real numbers  $\lambda_r$,  and setting $\delta_r=\min_{s\neq r}|\lambda_r-\lambda_s|$,
\begin{equation}\label{eqn:MontgomeryVaughan}
\frac{\Vert \hat{f} \Vert_{\infty}^2}{t}\int_t^{2t}\left|\sum_r a_re^{\ii\lambda_r s}\right|^2\rd s\leq
\sum_r|a_r|^2\left(1+\frac{c}{t\delta_r}\right)
\end{equation}
for some universal $c>0$.
Consequently, in our situation, taking $\lambda_p=2\log p$, and bounding uniformly $\hat f$, we get
$$
\E\left|\frac{1}{\lambda_t}\sum_{p\in\mathcal{P}}\frac{\log p}{p}\hat f\left(\frac{\log p^2}{\lambda_t}\right)p^{2\ii \omega t}
\right|^2
\leq
\frac{1}{\lambda_t^2}\sum_{p\in\mathcal{P},p\leq t}\frac{(\log p)^2}{p^2}
\left(1+\frac{c p}{2 t}\right)\to 0,
$$
where we just used $|\log p_1-\log p_2|>2 p_1^{-1}$ for prime numbers $p_1<p_2$.
\end{proof}


Concerning the terms 
$n=p$ appearing in Proposition \ref{prop:appExpl}, the following lemma computes the asymptotics  of the diagonal terms from the second  moment
for a fixed function $f$.
This will be the asymptotics of the  variance.

\begin{lemma}\label{Second Moment} 
Let $b_{pt}=\lambda_t^{-1}\Lambda_u(p)/\sqrt{p}\,\hat f(\log p/\lambda_t)$.
Suppose $\xi\hat{f}(\xi)^2$ and $(\xi\hat{f}(\xi)^2)'$ have the asymptotic bound $O(\xi^{-1})$ as $\xi\rightarrow\pm\infty$. Then as $t\rightarrow\infty$, for $u=t^{1/2}$,
$$
\sum_{p\in\mathcal{P}} \vert b_{pt}\vert^2 = (1+o(1))\int_0^{(\log t)/(2\lambda_t)} \xi|\hat{f}(\xi)|^2\rd \xi + \OO\left(\int_{(\log t)/(2\lambda_t)}^{(\log t)/\lambda_t} \xi |\hat{f}(\xi)|^2\rd \xi\right).
$$
\end{lemma}
\begin{proof} This lemma relies on a simple  asymptotic estimate based on the prime number theorem.
Let $p_k$ denote the $k$th prime, $q_k$ denote $\log p_k$, with $q_0=0$ by convention, and
$\Delta_k=q_k-q_{k-1}$. First consider the sum over $1\leq p \leq t^{1/2}$. By the mean value theorem,
$$
\left| \int_{q_{k-1}/\lambda_t}^{q_k/\lambda_t} \xi |\hat{f}(\xi)|^2\rd \xi - \frac{\Delta_k}{\lambda_t}\frac{q_k}{\lambda_t}|\hat{f}\left(\frac{q_k}{\lambda_t}\right)|^2 \right| \leq \text{Var}(\xi|\hat{f}(\xi)|^21_{[q_{k-1}/\lambda_t,q_k/\lambda_t]}(\xi))\frac{\Delta_k}{\lambda_t}.
$$
which implies
$$
\left| \int_{0}^{(\log t)/(2\lambda_t)}\xi |\hat{f}(\xi)|^2 \rd \xi - \frac{1}{\lambda_t^2}\sum_{p_k<t^{1/2}} q_k \Delta_k |\hat{f}\left( \frac{q_k}{\lambda_t} \right)|^2  \right| \leq \sum_k \text{Var}(\xi|\hat{f}(\xi)|^21_{[q_{k-1}/\lambda_t,q_k/\lambda_t]}(\xi))\frac{\Delta_k}{\lambda_t}.
$$
Since the derivative of $w|\hat{f}(w)|^2$ is bounded by a constant $M$, then the right hand side is bounded by $\sum_k M\Delta_k^2/\lambda_t^2$, which converges to $0$.

Moreover, using summation by parts, and  letting $\pi$ denote the usual prime-counting function,
\begin{multline*}
\frac{1}{\lambda_t}\sum_{k=1}^{\pi(t^{1/2})} \frac{q_k}{\lambda_t} |\hat{f}\left( \frac{q_k}{\lambda_t} \right)|^2(\Delta_k - k^{-1}) 
=\frac{1}{\lambda_t}\frac{q_{\pi(t^{1/2})}}{\lambda_t}|\hat{f}\left(\frac{q_{\pi(t^{1/2})}}{\lambda_t}\right)|^2(q_{\pi(t^{1/2})}-\log \pi(t^{1/2})+\OO(1))\\
-\frac{1}{\lambda_t}\sum_{k=1}^{\pi(t^{1/2})}(q_k-\log k)\left[ \frac{q_{k+1}}{\lambda_t}|\hat{f}\left( \frac{q_{k+1}}{\lambda_t}\right)|^2 - \frac{q_k}{\lambda_t}\hat{f}|\left(\frac{q_k}{\lambda_t} \right)|^2  \right]
\end{multline*}
Using the prime number theorem and $\xi\hat{f}(\xi)^2=O(\xi^{-1})$, the first term is bounded above by 
$$
c\,\frac{q_{\pi(t^{1/2})}-\log\pi(t^{1/2})}{q_{\pi(t^{1/2})}}+\oo(1)=c\,\frac{\log\log\pi(t^{1/2})}{\log \pi(t^{1/2})},
$$
which converges to $0$. Now look at the second term. Using $(\xi\hat{f}(\xi)^2)'=\OO(\xi^{-1})$, the term in brackets can be bounded, so there is the upper bound 
$$
c\,\frac{1}{\lambda_t}\sum_{k=1}^{\pi(t^{1/2})}\log\log k \frac{\lambda_t}{q_k}\frac{ \Delta_{k+1}}{\lambda_t}.
$$
Using the well-known result on prime gaps, $p_{k+1}-p_k<p_k^{\theta}$ for sufficiently large $k$ and for some $\theta<1$,
$$
q_{k+1} < q_k + \log(1+p_k^{\theta-1})<q_k+4p_k^{\theta-1}<q_k+8k^{\theta-1}.
$$
Thus the upper bound
$$
c\,\frac{1}{\lambda_t}\sum_{k=1}^{\pi(t^{1/2})}\frac{\log\log k}{k^{2-\theta}\log k},
$$
holds, which also converges to $0$. 

The sum over $t^{1/2}\leq p\leq t$ follows from a similar argument and the fact that $\Lambda_{t^{1/2}}(p)=\log t - \log p \leq \log p$.
\end{proof}

Our proof of Theorem \ref{thm:1}  and Theorem \ref{thm:2} relies on a mollification $f_\e$ of $f$  in order to apply the approximate explicit formula, Proposition \ref{prop:appExpl},  to the following result from \cite{Bou2010} (using an idea from \cite{Sou2009}).  

\begin{proposition}\label{prop:CLT}
Let  $a_{pt}$ ($p\in\mathcal{P},t\in\RR^+$) be given complex numbers, such that $\sup_{p}|a_{pt}|\to 0$,
$\sum_p |a_{pt}|^2\to\sigma^2$ as $t\to\infty$.
Assume the existence of some $(m_t)$ with $\log m_t/\log t\to 0$ and
\begin{equation}\label{eqn:TailCondition}
\sum_{p>m_t} |a_{pt}|^2 \left(1+\frac{p}{t}\right)\underset{t\to\infty}{\longrightarrow}0.
\end{equation}
Then, if $\omega$ is a uniform random variable on $(1,2)$,
$$\sum_{p\in\mathcal{P}} a_{pt} p^{-\ii \omega t}\overset{{\rm(weakly)}}{\longrightarrow} \sigma \mathcal{N}$$
as $t\to\infty$, $\mathcal{N}$ being a standard complex normal variable.
\end{proposition}

\begin{proof}[Proof of Theorem \ref{thm:2}]
Let $\phi_\e(x)=\frac{1}{\e}\phi\left(\frac{x}{\e}\right)$ be  a bump function, and
$f_\e=f*\phi_\e$. Moreover, remember that we defined $\sigma_t^2=\int_{-(\log t)/\lambda_t}^{(\log t)/\lambda_t}|\xi||\hat f(\xi)|^2$.
We know that $\sigma_t\to \infty$ as $t\to\infty$. We will choose $\e=\e_t$ by the end of this proof, and use $u=t^{1/2}$.

{\it First step}. The difference
$\sigma_t^{-1}\sum_\gamma (f_\e(\gamma_t)-f(\gamma_t))$ converges to 0 in probability
if $\e\ll\frac{\lambda_t}{\log t}\sigma_t$. Indeed
\begin{multline*}
\E\left|f_\e(\gamma_t)-f(\gamma_t)\right|\leq
\e^{-1}\E\int |f(\gamma_t-y)-f(\gamma_t)|\phi(y/\e)\rd y
\leq  c\e^{-1}\int_1^2\rd\omega\int_0^\e\rd y\left|\int_{\gamma_t-y}^{\gamma_t}|\rd f(u)|\right|\\
\leq c \int_1^2\rd\omega\int_{\gamma_t-\e}^{\gamma_t}|\rd f(u)|
\leq c \int|\rd f(u)|\int_{-\frac{u+\e}{t\lambda_t}+\frac{\gamma}{t}\leq\omega\leq -\frac{u}{t\lambda_t}+\frac{\gamma}{t},|\gamma|\leq2t+\frac{|u|}{\lambda_t}}\rd\omega\\
\leq c\frac{\e}{t\lambda_t}    \int \1_{|\gamma|\leq2t+\frac{|u|}{\lambda_t}}|\rd f(u)|.
\end{multline*}
Hence, using (\ref{eqn:Nt}), we get that
$
\sum_\gamma \E\left|f_\e(\gamma_t)-f(\gamma_t)\right|
\leq\e \frac{\log t}{\lambda_t}\int(1+|u\log u|)|\rd f(u)|
$, and this last integral is bounded from the hypothesis (\ref{eqn:cond2}).

{\it Second step.} Let $Y_\e=\frac{1}{\sigma_t\lambda_t}\sum_{p\in\mathcal{P}}\frac{\Lambda_u(p)}{\sqrt{p}}\hat f_\e\left(\frac{\log p}{\lambda_t}\right)p^{\ii \omega t}$ and
$Y=\frac{1}{\sigma_t\lambda_t}\sum_{p\in\mathcal{P}}\frac{\Lambda_u(p)}{\sqrt{p}}\hat f\left(\frac{\log p}{\lambda_t}\right)p^{\ii \omega t}$.
Then  we have  $\|Y_\e-Y\|_{L^2}=\OO\left(\e\frac{\log t}{\lambda_t}\right)$ as
$t\to\infty$.
Indeed, we can bound
$\|Y_\e-Y\|^2_{L^2}$ by the diagonal terms in the expansion  because of Montgomery-Vaughan inequality,
stated in (\ref{eqn:MontgomeryVaughan}).
In our case, taking $\lambda_p=\log p$, $a_p=\frac{1}{\sigma_t\lambda_t}\frac{\Lambda_u(p)}{\sqrt{p}}\left(\hat f_\e\left(\frac{\log p}{\lambda_t}\right)-\hat f\left(\frac{\log p}{\lambda_t}\right)\right)$ using that $|\hat f_\e(u)-\hat f(u)|\leq c\, u\e|\hat f(u)|$, we get
$$
\|Y_\e-Y\|^2_{L^2}\leq c \left(\e\frac{\log t}{\lambda_t}\right)^2\frac{1}{\sigma_t^2}\sum_p
\left|\frac{1}{\lambda_t}\frac{\Lambda_u(p)}{\sqrt{p}}\hat f\left(\frac{\log p}{\lambda_t}\right)\right|^2
$$
and this last sum is asymptotically equivalent to $\sigma_t^2$, by Lemma \ref{Second Moment}.

{\it Third step.} We can easily find some  $m_t$ so that $\log m_t=\oo(\log t)$ and the tail condition $(\ref{eqn:TailCondition})$
is satisfied, for
$$a_{pt}=\frac{1}{\sigma_t\lambda_t}\frac{\Lambda_u(p)}{\sqrt{p}}\hat f\left(\frac{\log p}{\lambda_t}\right).$$
Indeed, as $f$ has bounded variation, $\hat f(x)=\OO(x^{-1})$, so
$$
\sum_{m_t<p<t}|a_{pt}|^2\leq\frac{1}{\sigma_t^2}\sum_{m_t<p<t} \frac{1}{p}\sim \frac{1}{\sigma_t^2}
(\log\log t-\log\log m_t).
$$
A possible choice is  $m_t=\exp(\log t/\sigma_t)$.

{\it Fourth step.} The error term (\ref{eqn:error}) in the approximate explicit formula for $f_\e$ can be controlled in the following way.
As $f$ has bounded variation, and $f''=\phi_\e''*f$,  it is a standard argument that
\begin{multline*}
\int|f_\e''(u)|\rd u=
\e^{-3}\int\left|\int\phi''\left(\frac{x}{\e}\right)(f(u-x)-f(u))\rd x\right|\rd u\\
\leq \e^{-3}\iint\left|\int_{u-x}^u|\rd f(v)|\right|\rd v  \1_{|x|<\e}\rd x   \rd u
\leq
\e^{-3}\int|\rd f(v)|\iint \1_{|x|<\e,v\in[u-x,u]}\rd x\rd u
\leq \e^{-1}\int|\rd f(v)|,
\end{multline*}
so the error term related to $f_\e''$  in Proposition \ref{prop:appExpl} is of order $\frac{\lambda_t}{\log t}\e^{-1}$.
All of the other error terms can be bounded in the same way, and have order at most $\frac{\lambda_t}{\log t}\e^{-1}$
as well.

{\it Conclusion.} From the previous steps, the conclusion of Theorem \ref{thm:2} holds if we can find some
$\e_t$ such that
$$
\frac{\lambda_t}{\sigma_t\log t}\ll\e_t\ll\frac{\lambda_t}{\log t}\sigma_t,
$$
which obviously holds for $\e_t=\lambda_t/\log t$.
Indeed, using the First and Second steps, to conclude we then just need 
\begin{align}
Y&\overset{{\rm(weakly)}}{\longrightarrow}\mathcal{N}\label{laststep1},\\
\frac{1}{\sigma_t}\sum_\gamma f_\e(\gamma_t)- Y_\e&\overset{{\rm(weakly)}}{\longrightarrow}0,\label{laststep2}
\end{align}
where $\mathcal{N}$ is a standard complex Gaussian random variable. The convergence (\ref{laststep1}) 
is a consequence of the Third step and Lemma \ref{Second Moment}, to apply Proposition \ref{prop:CLT}.
The convergence (\ref{laststep2}) holds thanks to Proposition \ref{prop:appExpl}, the Fourth step and Lemma \ref{lem:largepowers}.
\end{proof}

\begin{proof}[Proof of Theorem \ref{thm:1}]
We closely follow the proof of Theorem \ref{thm:2}, except that now all of the errors due to the mollification need to
vanish without normalization. This is possible  because of the extra regularity assumptions (\ref{eqn:cond1}) we chose to  assume in Theorem
\ref{thm:1}.

The error from the first and second steps are controlled exactly in the same way, and they
will be negligible if $\e\ll\frac{\lambda_t}{\log t}$. The error from the third step vanishes
if $\frac{\lambda_t}{\log t}\|f_\e''\|_{L^1}\to 0$ (for the $f_\e''$ term in  (\ref{eqn:error}), for example). As in the proof of Theorem \ref{thm:2},
$\|f_\e''\|_{L^1}$ can be bounded by the total variation of $f'$, so this error goes to 0  anyway.
\end{proof}

\end{document}